\UseRawInputEncoding
\documentclass[review,hidelinks,onefignum,onetabnum]{siamart250106}

\usepackage{lipsum}
\usepackage{amsfonts}
\usepackage{graphicx}
\usepackage{epstopdf}
\usepackage{algorithmic}
\ifpdf
  \DeclareGraphicsExtensions{.eps,.pdf,.png,.jpg}
\else
  \DeclareGraphicsExtensions{.eps}
\fi

\newsiamremark{remark}{Remark}
\newsiamremark{hypothesis}{Hypothesis}
\crefname{hypothesis}{Hypothesis}{Hypotheses}
\newsiamthm{claim}{Claim}

\headers{Scalable Fixed-Point Framework for HJ PDEs}{Y. Park and S. Osher}

\title{Scalable Fixed-Point Framework for High-Dimensional Hamilton–Jacobi Equations\thanks{Submitted to the editors DATE 
\funding{S. Osher and Y. Park were supported by DARPA HR00112590074 and DoE DE-SC0026262.
S. Osher was also supported by NSF 2208272.}}}

\author{Yesom Park\thanks{Department of Mathematics, University of California, Los Angeles 
  (\email{yeisom@math.ucla.edu}).}
\and Stanley Osher\thanks{Department of Mathematics, University of California, Los Angeles
  (\email{sjo@math.ucla.edu}).}}

\usepackage{amsopn}

\ifpdf
\hypersetup{
  pdftitle={Neural Implicit Solution Formula for HJ PDEs},
  pdfauthor={Y. Park and S. Osher}
}
\fi

\usepackage{subfigure}
\usepackage{booktabs}

\usepackage{empheq}

\usepackage{multirow}

\newcommand{\cU}{\mathcal{U}}

\newcommand{\bR}{\mathbb{R}}

\newcommand{\dV}[1][u]{\cU\left(\Omega;\bR^{d_v}\right)}
\newcommand{\dVt}[1][u]{\cU\left(\left[0,\infty\right)\times\Omega;\bR^{d_v}\right)}

\newcommand{\bx}{\mathbf{x}}

\begin{document}

\maketitle
\begin{abstract}
We propose a novel, mesh--free, and gradient--free fixed--point approach for computing viscosity solutions of high-dimensional Hamilton--Jacobi (HJ) equations. By leveraging the Hopf--Lax formula, our approach iteratively solves the associated variational problem via a Picard iteration, enabling efficient evaluation of both the solution and its corresponding control without relying on grids, characteristics, or differentiation.
We demonstrate the practical efficacy and scalability of the approach through numerical experiments in up to 100 dimensions, including control problems and non-smooth solutions. Our results show that the proposed scheme achieves high accuracy, is highly efficient, and exhibits computational times that are largely independent of dimensionality, highlighting its suitability for high-dimensional problems.
\end{abstract}



\section{Introduction}


Hamilton--Jacobi partial differential equations (HJ equations) play a central role in numerous areas of mathematics, physics, and engineering, including optimal control \cite{evans1984differential, sideris2005efficient, bansal2017hamilton}, mechanics \cite{denman1973solution, de2008linear}, and the study of dynamic systems \cite{khanin2010particle, rajeev2008hamilton}. They provide a rigorous framework for modeling systems governed by physical laws and have found applications in diverse domains such as geometric optics \cite{osher1993level, minano2006hamilton}, computer vision \cite{caselles1997geodesic, gilboa2009nonlocal, osher2007geometric}, robotics \cite{lin1998optimal, lewis2003robot, bansaL2020hamilton}, trajectory optimization \cite{delahaye2014mathematical, parzani2018hamilton}, traffic flow modeling \cite{imbert2013hamilton, lavaL2013hamilton}, and financial strategies \cite{forsyth2007numerical, cao2009optimal}. Many of these applications naturally lead to \emph{high-dimensional} HJ equations, where the state space involves multiple degrees of freedom or coupled system variables. The inherent high dimensionality of these problems poses a severe challenge for classical numerical methods.

Traditional grid-based methods, including essentially non-oscillatory (ENO) and weighted ENO (WENO) schemes \cite{osher1991high, jiang2000weighted, bryson2003high, qiu2005hermite}, semi-Lagrangian methods \cite{falcone2002semi, cristiani2007fast, falcone2013semi}, and level set approaches \cite{osher1988fronts, osher1993level, osher2004level, mitchell2004computing, ansari2013application}, achieve high accuracy in low-dimensional settings but rely heavily on discretizing the state space. As the dimensionality increases, the number of grid points required for accurate approximations grows exponentially, leading to prohibitive computational costs. Consequently, these methods become infeasible for many high-dimensional control problems, where the state space can easily exceed ten dimensions. 
 To address this \emph{curse of dimensionality}, alternative approaches have been proposed, including max-plus algebra-based methods \cite{mceneaney2006max, akian2008max, fleming2000max} and Hopf--Lax formula-based methods \cite{darbon2016algorithms, chow2017algorithm, chow2019algorithm, chen2024hopf}, which avoid explicit grid discretization and allow parallel computation by solving optimization problems at each point, though they are generally limited to specific problem classes and often require convexity or other structural assumptions.
 
Recent advances in neural network-based approaches \cite{sirignano2018dgm, yu2018deep, raissi2019physics, lu2019deeponet, li2020fourier, yang2023context} have leveraged neural networks to store solution information in their parameters, offering scalable representations that can partially mitigate the dimensionality challenge inherent to high-dimensional HJ equations. In particular, the development of specialized neural network architectures tailored to express variational or representation formulas for specific classes of HJ equations has shown promising results \cite{darbon2020overcoming, darbon2021some, darbon2023neural}. Complementarily, a new method based on characteristics has been proposed to efficiently compute viscosity solutions via an implicit formula \cite{park2025implicit}.  
Despite these advancements, convergence and error guarantees of these approaches are often absent or heuristic, which limits the reliability of the computed solutions. Moreover, the repeated computation of derivatives through automatic differentiation for backpropagation incurs substantial computational and memory costs as the problem dimension increases. These factors constrain the practical scalability and efficiency of neural solvers for extremely high-dimensional HJ equations, particularly in real-world applications where both accuracy and computational tractability are critical.

\begin{table}[t]
\centering
\caption{Comparison of computational paradigms for solving Hamilton--Jacobi equations. 
The proposed fixed-point method performs \emph{direct local evaluation}, 
deep learning methods represent a \emph{global approximation} requiring prior training, 
and WENO is an \emph{evolutionary solver} advancing solutions causally in time.}
\vspace{4pt}
\resizebox{\textwidth}{!}{%
\begin{tabular}{lccccc}
\toprule
\textbf{Method} & \textbf{Computation Type} & \textbf{Precomputation} & \textbf{Evaluation Cost} & \textbf{Scalability} & \textbf{Convergence}\\
\midrule
\textbf{Grid-based methods} & Evolutionary & Grid generation  & Exponential in grid size & Exponential in $d$ & Guaranteed\\
\textbf{Deep learning methods} & Global approximation & Heavy training & Low (forward network pass) & Quadratic w.r.t.\ $d$ & No \\
\textbf{Fixed-point (Ours)} & Local, mesh-free & None & Linear in iteration count & Nearly $d$-independent & Guaranteed \\
\bottomrule
\end{tabular}}
\label{tab:method_comparison}
\end{table}

In this work, we introduce a novel \emph{fixed-point iteration framework} for solving high-dimensional HJ equations via the Hopf--Lax formula. Unlike traditional grid-based or neural network-based methods, our approach is \emph{mesh-free}, \emph{causality-free}, and \emph{derivative-free}, making it particularly suitable for large state dimensions. By reformulating the Hopf--Lax formula as a fixed-point problem, we obtain a simple iterative scheme that converges to the unique global minimizer, enabling direct computation of the viscosity solution and associated optimal controls without spatial discretization or repeated differentiation.  
Importantly, the algorithm is straightforward to implement, requiring only basic iterative updates, which makes it highly accessible for practical use.

We also establish rigorous convergence and error bounds under standard convexity and Lipschitz continuity assumptions, covering a wide range of practical optimal control and high-dimensional Hamilton--Jacobi problems. Within this framework, the method achieves unprecedented scalability while avoiding grids, characteristic tracing, and derivative evaluations. Moreover, in scenarios where overlapping characteristics produce kinks and the fixed point is no longer unique, we propose a randomly sampled \emph{multiple initialization} strategy to explore all fixed points and reliably identify the one corresponding to the viscosity solution. Numerical experiments in up to 100 dimensions, including non-smooth and control problems, demonstrate high accuracy, efficiency, and near dimension-independent computational cost.  

Table~\ref{tab:method_comparison} highlights the distinctions between our fixed-point approach, deep learning methods, and classical grid-based solvers. In contrast to deep learning methods, which require heavy global training, and grid-based solvers, which are evolutionary and causal from $t=0$, the fixed-point method supports \emph{direct local evaluation} at any point $(x,t)$, with guaranteed convergence and nearly dimension-independent cost.  
Overall, our framework complements the limitations of existing classical numerical schemes and modern deep learning approaches: it combines \emph{dimension-robust scalability} with \emph{provable accuracy}, providing a theoretically sound and practically efficient methodology for high-dimensional HJ equations.


\section{Method}
\subsection{Fixed-Point Representation via the Hopf--Lax Formula}
Consider the Hamilton--Jacobi (HJ) equation
\begin{equation}\label{eq:hj}
\begin{cases}
    u_t + H(\nabla u) = 0,\\
    u( x ,0) = g( x ),
\end{cases}
\end{equation}
where \(H:\mathbb{R}^d\rightarrow\mathbb{R}\) is a convex (or concave) Hamiltonian and \(g:\mathbb{R}^d\rightarrow\mathbb{R}\) is a Lipschitz continuous initial condition. It is well-known that, under suitable regularity assumptions, the unique viscosity solution of \eqref{eq:hj} can be expressed via the Hopf--Lax formula:
\begin{equation}\label{eq:hopf_lax}
    u(x,t) = \inf_{y \in \mathbb{R}^n} \Big\{ t\,H^\ast\Big(\frac{x-y}{t}\Big) + g(y) \Big\},
\end{equation}
where \(H^\ast\) denotes the Legendre transform of \(H\).  

While the Hopf--Lax formula provides a variational characterization of the solution, its practical evaluation requires computing the global minimizer \(y^*\). By formally differentiating the variational expression, we arrive at the following fixed-point relation for the minimizer:
\begin{equation}\label{eq:fixed_pts}
    y = x - t\, \nabla H(\nabla g(y)),
\end{equation}
which naturally motivates a Picard fixed-point iteration for approximating $y^*$. Specifically, the iterative procedure
\begin{equation}\label{eq:fp_iter}
    y^{(k+1)} = x - t\, \nabla H(\nabla g(y^{(k)})), \quad k = 0,1,2,\dots
\end{equation}
can be implemented in practice with a prescribed maximum number of iterations and convergence tolerance. The detailed algorithmic implementation is summarized in Algorithm~\ref{alg:fixed_point}. 
This fixed-point perspective establishes a bridge between the variational formulation of HJ equations and efficient computational realization, enabling scalable evaluation even in high-dimensional settings such as optimal control.

\begin{algorithm}[H]
\caption{Fixed-Point Iteration for Hamilton--Jacobi Equation}
\label{alg:fixed_point}
\begin{algorithmic}[1]
\REQUIRE $x \in \mathbb{R}^d$, $t>0$, $g:\mathbb{R}^d \to \mathbb{R}$, $H:\mathbb{R}^d \to \mathbb{R}$, 
$K_{\max} \in \mathbb{N}$ (maximum iterations), $\varepsilon > 0$ (tolerance)
\ENSURE $y^* \in \mathbb{R}^d$, approximate minimizer of the Hopf--Lax formula

\STATE Initialize $y^{(0)} \gets x$
\FOR{$k = 0$ \textbf{to} $K_{\max}-1$}
    \STATE $y^{(k+1)} \gets x - t \, \nabla H(\nabla g(y^{(k)}))$
    \IF{$\|y^{(k+1)} - y^{(k)}\|_2 < \varepsilon$}
        \STATE \textbf{break}
    \ENDIF
\ENDFOR
\STATE $y^* \gets y^{(k+1)}$
\RETURN $y^*$
\end{algorithmic}
\end{algorithm}

\paragraph{Control-theoretic interpretation.}
The fixed-point relation \eqref{eq:fixed_pts} admits a natural interpretation from the viewpoint of optimal control.  
Recall that the Hamilton--Jacobi equation can be equivalently expressed as
\[
u_t + H(\nabla u) = u_t + \sup_{q}\{\nabla u^\top q - L(q)\} = 0,
\]
where \(L\) is the Legendre transform of \(H\).  
The corresponding value function of the optimal control problem is given by
\begin{equation}\label{eq:oc_value_ft}
    u( x ,t) = \inf_{q(\cdot)} \Bigg\{ \int_0^t L(q(s))\,\mathrm{d}s + g( y (0))
    :  y (t)= x ,\; \dot{ y }(s)=q(s) \Bigg\}.
\end{equation}
Under standard convexity and differentiability assumptions, the minimizer \(y^\ast\) obtained from the fixed-point iteration corresponds to the optimal initial state \( y (0)\) in \eqref{eq:oc_value_ft}, and the associated optimal control is given by
\begin{equation}\label{eq:control_from_y}
q^\ast = \frac{x - y^\ast}{t}.
\end{equation}
Hence, the proposed fixed-point scheme not only enables efficient computation of the viscosity solution \(u(x,t)\), but also provides direct access to the corresponding optimal control, without requiring characteristic tracing or grid-based discretization.

\paragraph{Advantages.}
Compared to traditional mesh-based numerical solvers (e.g., WENO schemes) or recent deep learning-based approaches, the proposed fixed-point method achieves significantly faster evaluation while maintaining high accuracy. 
Its key advantages include:
\begin{itemize}
    \item \textbf{Mesh-Free and Scalable:} No spatial discretization or characteristic tracing is required, making the method well-suited for high-dimensional problems.
    \item \textbf{Derivative-Free:} The approach avoids explicit derivative computations, resulting in low computational overhead and excellent scalability.
    \item \textbf{Casuality-free and Highly Parallelizable:}
    The fixed-point updates do not rely on causality from neighboring grid points, resulting in a causality-free scheme that enables embarrassingly parallel computation across all spatial locations and significantly reduces wall-clock time on modern hardware.
    \item \textbf{Easy Implementation:} The algorithm consists of simple iterative updates, making it straightforward to implement and highly accessible for practical use.
    \item \textbf{Highly Fast in High Dimensions:} The pointwise fixed-point evaluation enables substantial computational speedups and maintaining efficiency even in very high-dimensional settings.
    \item \textbf{Control Recovery:} The optimal control associated with the Hamilton--Jacobi solution is directly obtained from the fixed-point minimizer \eqref{eq:control_from_y}.
    \item \textbf{Wide applicability:} The scheme accommodates a broad class of convex Hamiltonians and convex, Lipschitz-differentiable initial data, and can be extended to certain nonsmooth settings.
\end{itemize}

\subsection{Convergence Analysis}
In this section, we establish rigorous convergence, error, and complexity guarantees for the proposed fixed-point iteration scheme~\eqref{eq:fp_iter}.
We derive quantitative bounds on both the minimizer error and the approximation error in the viscosity solution, and further analyze the iteration complexity required to achieve a prescribed accuracy.
Together, these results provide a comprehensive theoretical foundation for the efficiency, stability, and accuracy of the proposed method.

We begin by presenting sufficient conditions that ensure the existence and uniqueness of the fixed point, as well as the global convergence of the iteration. Throughout this section, $\|\cdot\|$ denotes the $L^2$ norm.

\begin{theorem}[Fixed-point convergence]\label{thm:fixed_pt}
Let \( H: \mathbb{R}^n \to \mathbb{R} \) and \( g: \mathbb{R}^n \to \mathbb{R} \) be continuously differentiable functions such that:
\begin{itemize}
    \item \( H \) is \emph{strictly convex} and has \( L_H \)-Lipschitz continuous gradient:
    \[
    \|\nabla H(p_1) - \nabla H(p_2)\| \leq L_H \|p_1 - p_2\|, \quad \forall p_1, p_2 \in \mathbb{R}^n.
    \]
    \item \( g \) is \emph{convex} and has \( L_g \)-Lipschitz continuous gradient:
    \[
    \|\nabla g(y_1) - \nabla g(y_2)\| \leq L_g \|y_1 - y_2\|, \quad \forall y_1, y_2 \in \mathbb{R}^n.
    \]
\end{itemize}
Define the fixed-point operator
\[
    F(y) := x - t\, \nabla H(\nabla g(y)).
\]
Then, for any \(x \in \mathbb{R}^n\) and any time step \(t>0\) satisfying
\[
    t\, L_H\, L_g < 1,
\]
the iteration
\[
    y^{(k+1)} = F(y^{(k)}), \qquad k = 0,1,2,\dots
\]
converges to the unique global minimizer of the Hopf--Lax formula \eqref{eq:hopf_lax}.

\end{theorem}


\begin{proof}
We aim to show that \( F(y) = x - t \nabla H(\nabla g(y)) \) is a contraction mapping.

Let \( y_1, y_2 \in \mathbb{R}^n \). Then:
\begin{align*}
\|F(y_1) - F(y_2)\| 
&= \left\| x - t \nabla H(\nabla g(y_1)) - \left( x - t \nabla H(\nabla g(y_2)) \right) \right\| \\
&= t \left\| \nabla H(\nabla g(y_1)) - \nabla H(\nabla g(y_2)) \right\| \\
&\leq t\, L_H \, \|\nabla g(y_1) - \nabla g(y_2)\| \\
&\leq t\, L_H\, L_g\, \|y_1 - y_2\|.
\end{align*}
The first inequality follows from the \(L_H\)-Lipschitz continuity of \(\nabla H\),
and the second from the \(L_g\)-Lipschitz continuity of \(\nabla g\).

Define \( \lambda := t \cdot L_H \cdot L_g \). Then:
\[
\|F(y_1) - F(y_2)\| \leq \lambda \|y_1 - y_2\|.
\]

Since \( \lambda < 1 \) by assumption, \( F \) is a contraction on the complete metric space \( \mathbb{R}^n \). By the Banach fixed-point theorem, there exists a unique fixed point \( y^* \in \mathbb{R}^n \) such that
\[
y^* = F(y^*) = x - t \nabla H(\nabla g(y^*)).
\]


We now show that this fixed point \( y^* \) is the unique minimizer of the Hopf--Lax function
\[
\phi(y) := t H^*\left( \frac{x - y}{t} \right) + g(y).
\]

Since \( H^* \) is the convex conjugate of a strictly convex and differentiable function \( H \), the Hopf-Lax function is differentiable.
Therefore, the critical point condition for \( \phi \) is:
\[
0 \in \nabla \phi(y) = - \nabla H^*\left( \frac{x - y}{t} \right) + \nabla g(y).
\]
Equivalently,
\[
\nabla g(y) = \nabla H^*\left( \frac{x - y}{t} \right)
\quad \Longleftrightarrow \quad y = x - t \nabla H(\nabla g(y)),
\]
where we used the identity \( \nabla H^* = (\nabla H)^{-1} \) under strict convexity.
Thus, any solution to the fixed-point equation is a critical point of \( \phi \), and vice versa.

Since \( H^* \) is strictly convex (as \( H \) is strictly convex), and \( g \) is convex, their sum \( \phi(y) \) is strictly convex. Hence, \( \phi \) has a unique critical point, which must be the unique global minimizer of \( \phi \).

Therefore, the fixed point \( y^* \) is the unique minimizer of the Hopf--Lax functional:
\[
u(x,t) = \min_y \left\{ t H^*\left( \frac{x - y}{t} \right) + g(y) \right\} = \phi(y^*).
\]
\end{proof}

By establishing the convergence of the fixed-point iteration, Theorem~\ref{thm:fixed_pt} provides a theoretical foundation for a \emph{gradient-free, mesh-free, and causality-free} approach to evaluating viscosity solutions of high-dimensional Hamilton--Jacobi equations.

We now proceed to quantify the relationship between the residual error and the resulting bounds on both the minimizer and the Hopf--Lax functional value, thereby providing explicit estimates for the approximation accuracy of the method.

\begin{theorem}[Error Analysis]\label{thm:residual_error}
Let \(H\) and \(g\) satisfy the assumptions of Theorem~\ref{thm:fixed_pt}, and let \(t>0\) satisfy \(L_F := t L_H L_g < 1\).  
Consider the iteration \(y^{(k+1)} = F(y^{(k)})\) with \(F(y) = x - t\,\nabla H(\nabla g(y))\),
and denote the minimizer by \(y^\ast\). Define the residual \(r^{(k)} := y^{(k+1)} - y^{(k)}\).
Then:
\begin{enumerate}
    \item (\textbf{Hopf--Lax Minimizer error})
    \[
    \|y^{(k)} - y^\ast\| \le \frac{1}{1 - L_F} \|r^{(k)}\|.
    \]
    \item (\textbf{Solution error})  
    If \(H^\ast\) is \(L_{H^\ast}\)-Lipschitz on the relevant domain, then
    \[
    |u(x,t) - u^{(k)}(x,t)| \le \frac{L_{H^\ast}/t + L_g}{1 - L_F}\, \|r^{(k)}\|,
    \]
    where \(u^{(k)}(x,t) := t\, H^\ast\!\big(\frac{x - y^{(k)}}{t}\big) + g(y^{(k)})\).
    \item (\textbf{Solution gradient error})  
    \[
    \|\nabla u(x,t) - \nabla u^{(k)}(x,t)\| 
\le L_g\| r^{(k)} \|.
    \]
\end{enumerate}
\end{theorem}

\begin{proof}
\textbf{1. Hopf--Lax minimizer error bound.} 
  By Theorem \ref{thm:fixed_pt}, the minimizer of the Hopf--Lax formula $y^\star$ is the unique fixed point of \(F\).
Since \(F\) is a contraction with constant \(L_F\), we have
\[
\|F(y) - F(y^*)\| \le L_F \|y - y^*\| \quad \forall y \in \mathbb{R}^n.
\]

By definition of the residual,
\[
r^{(k)} = y^{(k+1)} - y^{(k)} = F(y^{(k)}) - y^{(k)}.
\]

Then,
\[
\begin{aligned}
\|y^{(k)} - y^*\| 
&= \|y^{(k)} - F(y^*)\| \\
&= \|y^{(k)} - F(y^{(k)}) + F(y^{(k)}) - F(y^*)\| \\
&\le \|r^{(k)}\| + \|F(y^{(k)}) - F(y^*)\| \\
&\le \|r^{(k)}\| + L_F \|y^{(k)} - y^*\|.
\end{aligned}
\]

Rearranging gives
\[
\|y^{(k)} - y^*\| \le \frac{1}{1 - L_F} \|r^{(k)}\|.
\]

\medskip
\textbf{2. Solution error bound.}  
Let
\[
u(x,t) = t H^*\Big(\frac{x - y^*}{t}\Big) + g(y^*), \quad u^{(k)}(x,t) = t H^*\Big(\frac{x - y^{(k)}}{t}\Big) + g(y^{(k)}).
\]

Using the Lipschitz continuity of \(H^*\) and \(g\),
\[
\begin{aligned}
|u(x,t) - u^{(k)}(x,t)| 
&\le t \left|H^*\Big(\frac{x - y^*}{t}\Big) - H^*\Big(\frac{x - y^{(k)}}{t}\Big)\right| + |g(y^*) - g(y^{(k)})| \\
&\le t \, L_{H^*} \frac{\|y^{(k)} - y^*\|}{t} + L_g \|y^{(k)} - y^*\| \\
&= (L_{H^*}/t + L_g) \|y^{(k)} - y^*\| \\
&\le \frac{L_{H^*}/t + L_g}{1 - L_F} \|r^{(k)}\|.
\end{aligned}
\]

\medskip
\textbf{3. Solution gradient error bound.}  
Taking derivative of Hopf--Lax formula with respect to $x$ gives
\[\nabla u(x,t) = \nabla H^*((x - y^*)/t).
\] 
By incorporating the implicit relation \eqref{eq:fixed_pts}, it
gives
\[\nabla u(x,t) = \nabla H^*(\nabla H(\nabla g(y^\ast)))=\nabla g(y^\ast),
\] 
where the last equality comes from the fact $ \nabla H^*=\nabla H^{-1}$.
Consequently, we have
\[
\begin{aligned}
\|\nabla u - \nabla u^{(k)}\| 
&= \Big\|\nabla g(y^{(k)}) - \nabla g(y^\ast)\Big\| \\
&\le L_g\| r^{(k)} \|.
\end{aligned}
\]

This completes the proof.
\end{proof}

Theorem~\ref{thm:residual_error} provides explicit quantitative bounds relating the residual norm to both the minimization error and the Hopf--Lax approximation error.  
In particular, controlling the error both in the Hopf--Lax minimizer $y^*$ and the solution $u$ is equivalent to controlling the error not only in the viscosity solution itself but also in its gradient.
From the perspective of control theory, this ensures that the error is controlled not only for the value function but also for the associated optimal control, thereby providing rigorous guarantees for both solution and control accuracy.
This equivalence highlights a key strength of the proposed methodology.
In practice, this result provides a rigorous and efficient termination criterion: once the residual \(\|r^{(k)}\|\) falls below a tolerance \(\varepsilon\),
the corresponding error in both the solution and its gradient (or, equivalently, in the value function and the associated optimal control) is guaranteed to remain within a controlled bound.

We now characterize the rate of convergence and the number of iterations required to reach a prescribed error threshold.

\begin{theorem}[Iteration Complexity]\label{thm:iteration_complexity}
Under the same assumption and notation in Theorem \ref{thm:residual_error},
for any initial guess \(y^{(0)} \in \mathbb{R}^d\),
the number of iterations required to achieve a residual norm
\(\|r^{(k)}\| \le \varepsilon\) satisfies
\[
    k \ge 
    \frac{\log\!\Big(\frac{(1 + L_F)\|y^{(0)} - y^*\|}{\varepsilon}\Big)}{-\log L_F}.
\]
In particular, the fixed-point iteration converges linearly with rate \(L_F\),
and thus the iteration complexity scales as
\[
    k = O\big(\log(1/\varepsilon)\big).
\]
\end{theorem}

\begin{proof}
Since \(F\) is a contraction with modulus \(L_F\), it holds that
\[
    \|y^{(k+1)} - y^*\| \le L_F \|y^{(k)} - y^*\|.
\]
By induction,
\[
    \|y^{(k)} - y^*\| \le L_F^k \|y^{(0)} - y^*\|.
\]
The residual satisfies
\[
    r^{(k)} = y^{(k+1)} - y^{(k)}
    = (y^{(k+1)} - y^*) - (y^{(k)} - y^*),
\]
and hence
\[
    \|r^{(k)}\| 
    \le \|y^{(k+1)} - y^*\| + \|y^{(k)} - y^*\|
    \le (1 + L_F)L_F^k \|y^{(0)} - y^*\|.
\]
To achieve \(\|r^{(k)}\| \le \varepsilon\), it suffices that
\[
    (1 + L_F)L_F^k \|y^{(0)} - y^*\| \le \varepsilon.
\]
Solving for \(k\) yields
\[
    k \ge 
    \frac{\log\!\Big(\frac{(1 + L_F)\|y^{(0)} - y^*\|}{\varepsilon}\Big)}{-\log L_F}.
\]
Since \(-\log L_F > 0\) and independent of \(\varepsilon\),
we have \(k = O(\log(1/\varepsilon))\), completing the proof.
\end{proof}

Combining the results of Theorems~\ref{thm:residual_error} and~\ref{thm:iteration_complexity}, the fixed-point iteration is guaranteed to converge linearly, with the convergence rate determined by the contraction modulus \(L_F\). While the formal upper bound on the number of iterations is independent of the problem dimension \(d\), in practice \(L_F\) may increase with \(d\), leading to a moderate slowdown in convergence. Nevertheless, this dimension-dependent effect is significantly less severe than the exponential scaling of computational cost observed in classical mesh-based schemes or recent deep learning-based approaches, rendering the fixed-point method highly scalable for high-dimensional HJ problems.

\subsection{Fixed Point Solver with Multiple Initializations}\label{sec:multi_initialization}
In general, the critical point of the Hopf--Lax formula is not necessarily unique. Even when both the Hamiltonian $H$ and the initial condition $g$ are smooth, the viscosity solution of the Hamilton--Jacobi equation may exhibit non-smooth behavior due to the intersection of characteristics over time, resulting in the formation of kinks. In such cases, multiple stationary points of the Hopf--Lax energy functional may exist, and equivalently, the fixed-point equation~\eqref{eq:fixed_pts} can admit multiple solutions. 

To robustly handle this situation, we adopt a multi-initialization strategy for the fixed-point iteration. The procedure is summarized as follows:

\begin{enumerate}
    \item \textbf{Multiple Random Initializations:} 
    Generate $N$ independent random initial guesses 
    $\{ y_0^{(i)} \}_{i=1}^N$, each sampled uniformly from the hypercube 
    $[-\alpha, \alpha]^d$ for a prescribed $\alpha > 0$.

    \item \textbf{Fixed-Point Iteration:} 
    For each initialization $y_0^{(i)}$, perform the iterative updates
    \[
        y_{k+1}^{(i)} = G(y_k^{(i)}, t, x),
    \]
    until convergence, i.e.,
    \[
        \| y_{k+1}^{(i)} - y_k^{(i)} \|_2 < \varepsilon,
    \]
    or until the maximum iteration count is reached.

    \item \textbf{Energy Evaluation:}
    For each converged solution $y^{(i)}$, compute the Hopf--Lax energy
    \[
        \mathcal{E}(y^{(i)}; x, t) 
        = t\,H^\ast\!\left( \frac{x - y^{(i)}}{t} \right)
        + g\!\left( y^{(i)} \right),
    \]
    where $H^\ast$ denotes the Legendre transform of $H$.

    \item \textbf{Selection of Optimal Fixed Point:}
    Among the converged fixed points, select
    \[
        y^\ast = \arg\min_{i} \mathcal{E}(y^{(i)}; x, t),
    \]
    which corresponds to the minimizer in the Hopf--Lax formula~\eqref{eq:hopf_lax}.
\end{enumerate}

This multi-initialization approach enables the algorithm to explore multiple potential fixed points and ensures that the selected solution corresponds to the true viscosity solution, even in the presence of non-smooth features such as kinks.

\begin{theorem}[Probabilistic Completeness of Multi-Initialization]
\label{thm:prob_complete}
Let $G:\mathbb{R}^d \to \mathbb{R}^d$ be a continuous map, and suppose that it admits 
a finite set of fixed points $\{y_1^*, \dots, y_M^*\}$ satisfying
\[
    y_i^* = G(y_i^*), \quad i = 1, \dots, M.
\]
Assume the following:
\begin{itemize}
    \item (\textbf{Local contraction}) For each $i$, there exists an open neighborhood 
    $\mathcal{B}_i \subset \mathbb{R}^d$ of $y_i^*$ such that 
    \[
        \| G(y) - G(z) \| \leq \kappa_i \|y - z\|, \qquad \forall y,z \in \mathcal{B}_i,
    \]
    with some $\kappa_i \in [0,1)$.
    Hence, each $\mathcal{B}_i$ is a basin of attraction under the iteration $y_{k+1}=G(y_k)$.

    \item (\textbf{Nonzero measure}) Each $\mathcal{B}_i$ has positive Lebesgue measure, i.e. 
    $\mu(\mathcal{B}_i) > 0$.

    \item (\textbf{Independent sampling}) Initial guesses 
    $y_0^{(1)}, \dots, y_0^{(N)}$ are drawn independently 
    from a probability distribution $\rho$ 
    that is absolutely continuous with respect to the Lebesgue measure, 
    supported on a compact set $\Omega \subset \mathbb{R}^d$ satisfying 
    $\mu(\Omega) > 0$.
\end{itemize}

Then, with probability one, as $N \to \infty$, the fixed-point iteration
\[
    y_{k+1}^{(j)} = G(y_k^{(j)}), \qquad j=1,\dots,N,
\]
converges to all fixed points $\{y_i^*\}_{i=1}^M$.
\end{theorem}

\begin{proof}
For each $i$, denote the probability that a single random initialization 
falls within the basin $\mathcal{B}_i$ as 
\[
    p_i = \int_{\mathcal{B}_i} \rho(y) \, dy.
\]
By assumption, $\rho$ is absolutely continuous and $\mu(\mathcal{B}_i) > 0$, 
hence $p_i > 0$.

The probability that none of the $N$ independent initializations falls within $\mathcal{B}_i$ is $(1-p_i)^N$. 
Therefore, the probability that at least one initialization lies in $\mathcal{B}_i$ is
\[
    1 - (1-p_i)^N.
\]
Since there are finitely many basins ($M<\infty$), the probability that all basins are covered is
\[
    \mathbb{P}_N = \prod_{i=1}^M \big( 1 - (1-p_i)^N \big).
\]
As $N \to \infty$, each term satisfies $\lim_{N\to\infty} (1 - (1-p_i)^N) = 1$, 
hence by continuity of finite products,
\[
    \lim_{N\to\infty} \mathbb{P}_N = 1.
\]
This proves that with probability approaching one, every basin $\mathcal{B}_i$ is sampled at least once.

Within each $\mathcal{B}_i$, the local contraction property ensures that 
for any initialization $y_0^{(j)} \in \mathcal{B}_i$, 
the iteration $y_{k+1}^{(j)} = G(y_k^{(j)})$ converges linearly to the corresponding 
fixed point $y_i^*$ by the Banach fixed point Theorem.

Finally, since each fixed point $y_i^*$ is stationary for the Hopf--Lax energy $\mathcal{E}$, 
the minimizer of $\mathcal{E}$ among $\{y_i^*\}_{i=1}^M$ corresponds to the global minimizer of the Hopf--Lax variational formula. 
Because all $\{y_i^*\}$ are found almost surely as $N\to\infty$, 
the selected minimizer $y^*$ equals the optimal solution almost surely.
\end{proof}

Theoretically, infinitely many initial points are required to guarantee convergence to all fixed points. In practice, however, a finite number of initial guesses suffices, as demonstrated in Section~\ref{sec:experiments}.

\section{Experimental Results}\label{sec:experiments}
In this section, we evaluate the accuracy and efficiency of the proposed fixed-point method through various numerical experiments. 
We demonstrate its scalability to high-dimensional problems, compare its performance with existing solvers in terms of accuracy and computational efficiency, and verify its ability to recover optimal controls as well as approximate nonsmooth viscosity solutions.

All experiments were conducted on a Linux server equipped with an Intel Core i7-6850K CPU (6 cores, 12 threads) at 3.6 GHz, 62 GB RAM, and an NVIDIA TITAN V GPU (12 GB). 
To ensure a fair comparison with existing deep learning approaches, both the WENO solver and the proposed fixed-point method were implemented in Python 3.

\subsection{High-dimensional problems}
\subsubsection{Burgers' equation}
We consider the Hamilton--Jacobi setup with 
\begin{equation}
H(p) = \frac{1}{2} \|p\|_2^2, \qquad g(x) = \frac{1}{2} \|x\|_2^2,
\end{equation}
which possesses a unique critical point at the origin. The exact solution is given by $u(x,t)=\frac{ \|x\|_2^2}{2(1+t)}$. Using the fixed-point iteration 
\[
y^{(k+1)} = x - t \nabla H(\nabla g(y^{(k)})).
\]

For comparison, we also consider the fifth-order WENO scheme as a representative grid-based method and ImplicitHJ \cite{park2025implicit}, a deep-learning-based approach for approximating viscosity solutions of HJ equations. Performance is evaluated in 1, 2, 3, 10, 50, and 100 dimensions. Table~\ref{tab:comparison} summarizes the $L^2$ and $L^\infty$ errors of the solution and its gradient, along with computational time and memory usage. All errors are computed over 128  points.  

WENO computations use a uniform grid of 128 points with a CFL number of 0.2. ImplicitHJ is trained following the original configuration in the reference \cite{park2025implicit}, and its computational time corresponds to the total training time. The fixed-point method is applied with a maximum of 1000 iterations and a convergence tolerance of $10^{-6}$ across all dimensions. For fair comparison, all methods are implemented in Python.

It should be noted that these three approaches are fundamentally different in their design principles and computational characteristics. Establishing a fully “fair” comparison is challenging, particularly when considering extremely high-dimensional problems. In setting up the experiments, careful attention was given to ensure that each method is evaluated in a manner representative of its intended use, while maintaining consistency in reported metrics across dimensions.

\begin{table}[t]
\centering
\caption{Comparison of numerical and gradient errors, computation time, and memory usage for different dimensions using WENO, ImplicitHJ, and Ours methods.}
\label{tab:comparison}
\renewcommand{\arraystretch}{1.15}
\setlength{\tabcolsep}{3.5pt}
\resizebox{\textwidth}{!}{%
\begin{tabular}{c|c|cccccc}
\hline
\textbf{Dim} & \textbf{Method} & $\mathbf{L^2}$ \textbf{Error} & $\mathbf{L^\infty}$ \textbf{Error} & $\mathbf{L^2}$ \textbf{Grad Error} & $\mathbf{L^\infty}$ \textbf{Grad Error} & \textbf{Time (s)} & \textbf{Mem (MB)} \\
\hline
1D & WENO & $1.84\times10^{-10}$ & $1.52\times10^{-4}$ & $3.06\times10^{-5}$ & $6.56\times10^{-3}$ & 0.34 & 0.05 \\
   & ImplicitHJ & $4.88\times10^{-10}$ & $6.91\times10^{-5}$ & $1.70\times10^{-6}$ & $4.08\times10^{-3}$ & 156081.30 & 0.27 \\
   & Ours & $2.43\times10^{-17}$ & $1.49\times10^{-8}$ & $6.37\times10^{-15}$ & $2.98\times10^{-7}$ & 0.034 & 0.001 \\
\hline
2D & WENO & $2.30\times10^{-6}$ & $8.62\times10^{-3}$ & $3.72\times10^{-5}$ & $1.58\times10^{-1}$ & 8.45 & 4.42 \\
   & ImplicitHJ & $1.57\times10^{-8}$ & $3.40\times10^{-4}$ & $1.22\times10^{-5}$ & $8.95\times10^{-3}$ & 175223.30 & 0.34 \\
   & Ours & $2.82\times10^{-16}$ & $5.96\times10^{-8}$ & $1.03\times10^{-14}$ & $2.98\times10^{-7}$ & 0.04 & 0.001 \\
\hline
3D & WENO & $7.22\times10^{-3}$ & $2.54\times10^{-2}$ & $5.46\times10^{-3}$ & $8.91\times10^{-1}$ & 1768.01 & 642.58 \\
   & ImplicitHJ & $1.52\times10^{-7}$ & $1.06\times10^{-3}$ & $7.70\times10^{-5}$ & $2.46\times10^{-2}$ & 177784.40 & 0.42 \\
   & Ours & $5.98\times10^{-16}$ & $1.19\times10^{-7}$ & $1.81\times10^{-14}$ & $2.98\times10^{-7}$ & 0.04 & 0.001 \\
\hline
10D & WENO & N/A & N/A & N/A & N/A & N/A & N/A \\
    & ImplicitHJ & $5.05\times10^{-6}$ & $6.62\times10^{-3}$ & $1.45\times10^{-3}$ & $3.83\times10^{-2}$ & 195997.48 & 0.95 \\
    & Ours & $7.90\times10^{-15}$ & $2.38\times10^{-7}$ & $6.10\times10^{-14}$ & $3.58\times10^{-7}$ & 0.04 & 0.001 \\
\hline
50D & WENO & N/A & N/A & N/A & N/A & N/A & N/A \\
    & ImplicitHJ & $4.71\times10^{-2}$ & $3.67\times10^{-1}$ & $2.58\times10^{-1}$ & $2.53\times10^{-1}$ & 258987.28 & 4.01 \\
    & Ours & $2.66\times10^{-13}$ & $1.91\times10^{-6}$ & $2.97\times10^{-13}$ & $3.58\times10^{-7}$ & 0.05 & 0.001 \\
\hline
100D & WENO & N/A & N/A & N/A & N/A & N/A & N/A \\
     & ImplicitHJ & $1.37\times10^{-1}$ & $4.91\times10^{-1}$ & $4.20\times10^{-1}$ & $5.46\times10^{-1}$ & 324249.12 & 7.82 \\
     & Ours & $1.23\times10^{-12}$ & $3.81\times10^{-6}$ & $6.06\times10^{-13}$ & $3.58\times10^{-7}$ & 0.05 & 0.001 \\
\hline
\end{tabular}}
\end{table}

The results in Table~\ref{tab:comparison} demonstrate the superior performance of the proposed fixed-point method across all dimensions. Notably, our method consistently achieves solution and gradient errors near machine precision, whereas WENO and ImplicitHJ exhibit increasing errors as the dimension grows. In 3D and higher, WENO becomes computationally infeasible, highlighting the curse of dimensionality inherent to grid-based approaches. ImplicitHJ, while scalable, requires extensive training and exhibits non-negligible errors in high dimensions.  

In contrast, the fixed-point method is \emph{mesh-free}, \emph{causality-free}, \emph{gradient-free}, and \emph{training-free}, making it inherently scalable to very high-dimensional problems. It computes both the solution and its gradient directly, with theoretically guaranteed accuracy, thereby eliminating approximation errors associated with discretization or neural-network training. Furthermore, computational time and memory usage remain nearly constant across dimensions, as the method does not rely on dense grids or iterative solvers that scale poorly with dimension, in agreement with the bounds established in Theorem~\ref{thm:residual_error}.

These results highlight the practical advantages of the fixed-point approach: it provides theoretically grounded, high-precision approximations for both the solution and its gradient, while circumventing the key limitations of conventional grid-based and deep-learning-based methods. The combination of \emph{dimension-independent efficiency} and \emph{provable accuracy} positions our method as a robust and scalable alternative for solving high-dimensional Hamilton--Jacobi equations.

Remarkably, even in the 100-dimensional case, the $L^2$-error of the computed solution was as low as $2.06 \times 10^{-14}$, demonstrating both the high accuracy and the scalability of the proposed fixed-point scheme in high-dimensional settings.

\subsubsection{LQR Control Problem}
We consider an Linear quadratic regulator (LQR)-type optimal control problem in which the cost functional is quadratic in both the control and the initial state:
\begin{equation}\label{eq:lqr_ocp}
u(x,t) = \min_{q(\cdot)} \left\{ \frac{1}{2} \int_0^t q(s)^\top R q(s)\, ds + \frac{1}{2} y(0)^\top Q y(0) \;\Big|\; 
\dot y(s) = q(s),\ y(t) = x \right\},
\end{equation}
where \(R\) and \(Q\) are symmetric positive definite matrices. The corresponding Hamiltonian is of the anisotropic quadratic form:
\[
H(p) = \frac{1}{2} p^\top R^{-1} p.
\]

The problem admits a closed-form analytic solution:
\begin{equation}
y^* = \big(I + t Q R^{-1} \big)^{-1} x, \qquad 
u(x,t) = \frac{1}{2} x^\top \big(I + t Q R^{-1}\big)^{-1} Q x.
\end{equation}

The proposed fixed-point iteration
\[
y^{(k+1)} = x - t \, y^{(k)} (R^{-1} Q)^\top
\]
is guaranteed to converge if 
\[
t < \frac{1}{\|R^{-1} Q\|_2},
\]
where \(\| \cdot \|_2\) denotes the spectral norm.

For a given dimension \(d\), we generated \(Q, R \in \mathbb{R}^{d \times d}\) as symmetric positive definite matrices using random Gaussian matrices. Specifically, we drew \(A, B \in \mathbb{R}^{d \times d}\) with independent standard normal entries and constructed
\[
Q = s_1 (A^\top A + I_d), \qquad 
R = s_2 (B^\top B + I_d),
\]
with positive scaling factors \(s_1, s_2 > 0\), ensuring \(Q, R \succ 0\).

We evaluated the fixed-point iteration at 100 randomly sampled points \((x,t)\) satisfying the contraction condition \(t < 1/\|R^{-1} Q\|_2\), and recorded the \(L^2\) and \(L^\infty\) errors as well as the computational time. Table~\ref{tab:lqr_results} summarizes the results for two independent random realizations of \(Q\) and \(R\). Both sets of experiments demonstrate that the fixed-point iteration achieves high accuracy and low computational cost across dimensions.

\begin{table}[t]
\centering
\caption{Fixed-point results for the LQR control problem with two independent random realizations of \(Q\) and \(R\). Errors are measured against the analytic solution.}
\label{tab:lqr_results}
\resizebox{\textwidth}{!}{%
\begin{tabular}{c|c|c|c|c|c|c}
\hline
\multirow{2}{*}{\textbf{Dim}} & \multicolumn{3}{c|}{\textbf{First Random Set}} & \multicolumn{3}{c}{\textbf{Second Random Set}} \\
\cline{2-7}
 & $L^2$ error & $L^\infty$ error & Time (s) & $L^2$ error & $L^\infty$ error & Time (s) \\
\hline
1   & $6.70 \times 10^{-17}$ & $1.19 \times 10^{-7}$ & 0.0041 & $4.31 \times 10^{-16}$ & $5.96 \times 10^{-8}$ & 0.0045 \\
10  & $3.26 \times 10^{-12}$ & $5.72 \times 10^{-6}$ & 0.0050 & $4.82 \times 10^{-12}$ & $7.63 \times 10^{-6}$ & 0.0052 \\
50  & $2.21 \times 10^{-9}$  & $1.22 \times 10^{-4}$ & 0.0069 & $1.87 \times 10^{-9}$ & $1.22 \times 10^{-4}$ & 0.0080 \\
100 & $3.84 \times 10^{-8}$  & $4.88 \times 10^{-4}$ & 0.0171 & $2.25 \times 10^{-8}$ & $3.66 \times 10^{-4}$ & 0.0104 \\
\hline
\end{tabular}}
\end{table}

\subsubsection{Cubic Hamiltonian Problem}
We consider the HJ equation with nonlinear Hamiltonian and initial condition
\[
H(p) = \frac{1}{3} \|p\|^3, \qquad g(x) =\| x\|^3,
\]
for which the corresponding Legendre transform is
\[
H^*(q) = \frac{2}{3} \|q\|^{3/2}.
\]

We evaluated the minimization for different dimensions \(d\) and computed the $L^2$ and maximum errors relative to the analytic solution. The results are summarized in Table~\ref{tab:hopf_lax_results}.

\begin{table}[t]
\centering
\caption{Fixed-point results for the cubical problem.}
\label{tab:hopf_lax_results}
\begin{tabular}{c|c|c|c}
\hline
\textbf{Dim} & \textbf{$L^2$ Error} & \textbf{$L^\infty$ Error} & \textbf{Time (s)}\\
\hline
1   & $4.53 \times 10^{-9}$ & $8.81 \times 10^{-8}$ & 0.0409\\
10  & $2.92 \times 10^{-8}$ & $1.25 \times 10^{-7}$ & 0.0458\\
50  & $1.73 \times 10^{-7}$ & $4.80 \times 10^{-7}$ & 0.0476\\
100 & $4.14 \times 10^{-7}$ & $1.15 \times 10^{-6}$ & 0.0561\\
\hline
\end{tabular}
\end{table}

The results demonstrate that the Hopf--Lax minimization achieves high accuracy across dimensions, with errors and computational time increasing only mildly as the dimension grows.






\subsection{Non-smooth solutions}
\subsubsection{Absolute-value Quadratic Form}
We consider Burgers' equation with a non-smooth initial condition
\[
g(x) = x |x|,
\] 
which can be equivalently formulated as the following optimal control problem:
\begin{equation}
u(x,t) = \inf_{q(\cdot)} \left\{ \frac{1}{2}\int_0^t q(s)^2 \, ds + g(y(0)) \;\Big|\; y(t) = x, \;\dot y(s) = q(s), \; 0 \le s \le t \right\}.
\end{equation}

The solution predicted using the fixed-point iteration is shown in Figure~\ref{fig:non_smooth}. As time progresses, one can observe the formation of kinks in the solution, supporting that the proposed method is capable of accurately capturing such sharp features arising from non-smooth initial data.

\begin{figure}
    \centering
    \includegraphics[width=0.95\linewidth]{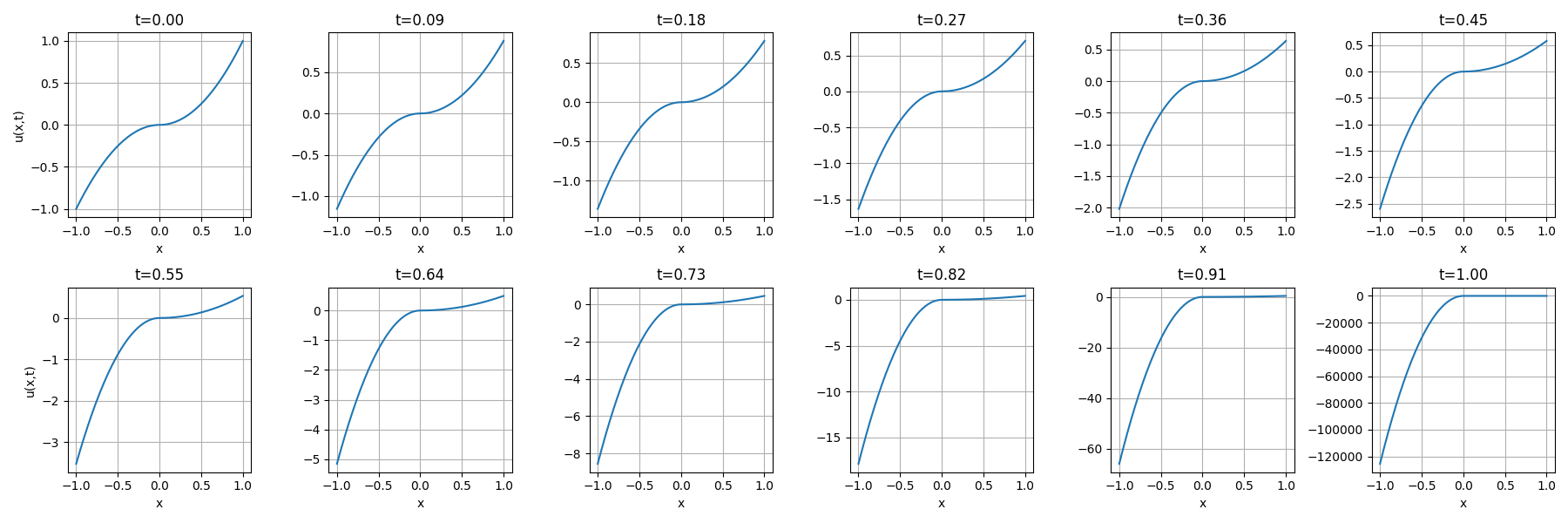}
    \caption{Solution of Burgers' equation with non-smooth initial condition $g(x) = x|x|$, computed via fixed-point iteration. Sharp gradients (kinks) develop as time evolves.}
    \label{fig:non_smooth}
\end{figure}

\subsubsection{Logarithmic Form}
We consider Burgers' equation with a non-smooth initial condition
\[
g(x) = x^2\log(2+ |x|).
\]

The solution predicted using the fixed-point iteration is shown in Figure~\ref{fig:non_smooth_log}. As time progresses, one can observe the formation of kinks in the solution, supporting that the proposed method is capable of accurately capturing such sharp features arising from non-smooth initial data.

\begin{figure}
    \centering
    \includegraphics[width=0.95\linewidth]{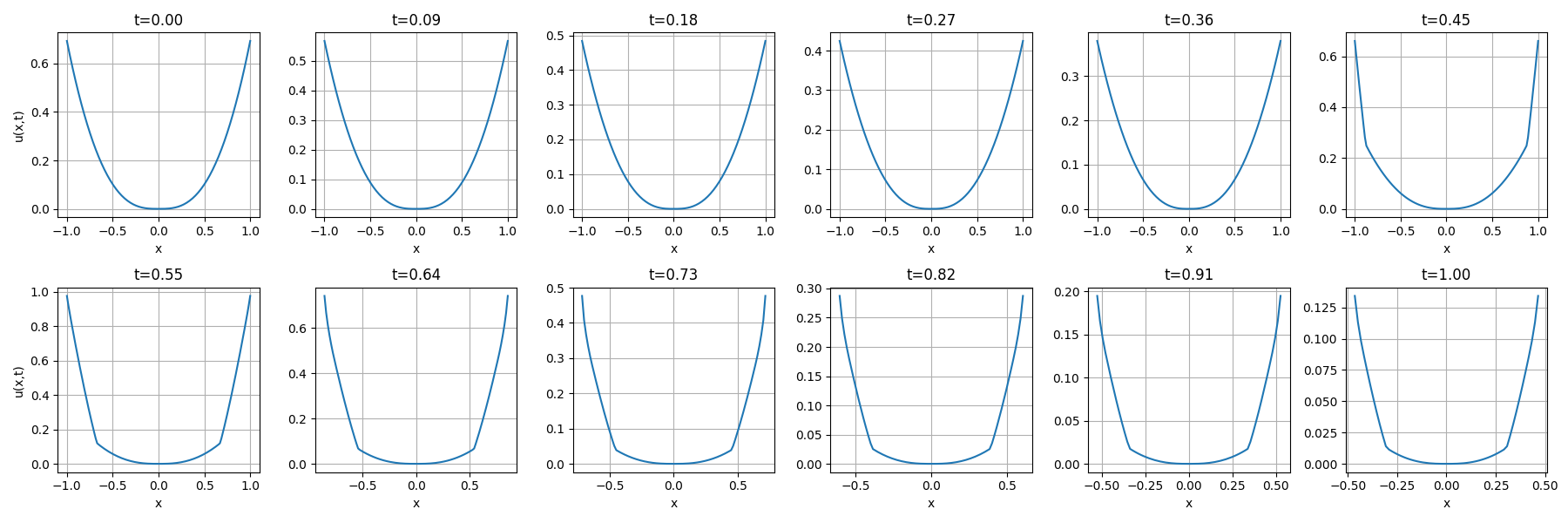}
    \caption{Solution of Burgers' equation with non-smooth initial condition $g(x) = x^2\log(1+|x|)$, computed via fixed-point iteration. Sharp gradients (kinks) develop as time evolves.}
    \label{fig:non_smooth_log}
\end{figure}

\subsection{Multiple Initialization}
This section investigates the effectiveness of multiple initialization scheme in handling kinks arising from intersecting characteristics.

\subsubsection{Steady Kink}
Consider the quadratic Hamiltonian
 $H\left(p\right) = \frac{1}{2}\left\Vert p\right\Vert_2^2$ and initial function $g\left(\bx\right) = -\left\Vert x\right\Vert_1$. 
The exact solution is 
\[
u\left(x,t\right)=-\left\Vert x\right\Vert_1-\frac{dt}{2},
\] where a steady kink occurs at \(x=0\). This arises because characteristics overlap in the region \(-t < x < t\), resulting in non-uniqueness of fixed points within this region. In this experiment, we examine whether the multiple initialization strategy proposed in Section \ref{sec:multi_initialization} is capable of effectively identifying the fixed point that corresponds to the viscosity solution among these multiple possibilities.

\begin{table}[ht]
\centering
\caption{Comparison of numerical and gradient errors, computation time, and memory usage for different dimensions using WENO, ImplicitHJ, and Ours methods on the steady kink example.}
\label{tab:comparison_steady_kink}
\renewcommand{\arraystretch}{1.15}
\setlength{\tabcolsep}{3.5pt}
\resizebox{\textwidth}{!}{%
\begin{tabular}{c|c|cccccc}
\hline
\textbf{Dim} & \textbf{Method} & $\mathbf{L^2}$ \textbf{Error} & $\mathbf{L^\infty}$ \textbf{Error} & $\mathbf{L^2}$ \textbf{Grad Error} & $\mathbf{L^\infty}$ \textbf{Grad Error} & \textbf{Time (s)} & \textbf{Mem (MB)} \\
\hline
1D & WENO & $1.51\times10^{-7}$ & $1.59\times10^{-3}$ & $2.17\times10^{-4}$ & $2.44\times10^{-3}$ & 0.35 & 0.07 \\
   & ImplicitHJ & $8.59\times10^{-6}$ & $1.72\times10^{-5}$ & $2.01\times10^{-5}$ & $3.15\times10^{-5}$ & 210241.37 & 52.91 \\
   & Ours & $1.54\times10^{-14}$ & $3.43\times10^{-6}$ & $2.15\times10^{-13}$ & $1.08\times10^{-6}$ & 25.97 & 3.86 \\
\hline
2D & WENO & $3.64\times10^{-5}$ & $1.02\times10^{-2}$ & $7.77\times10^{-3}$ & $7.45\times10^{-1}$ & 4.50 & 4.42 \\
   & ImplicitHJ & $1.10\times10^{-4}$ & $2.05\times10^{-4}$ & $3.15\times10^{-4}$ & $5.10\times10^{-4}$ & 168770.98 & 52.99 \\
   & Ours & $2.82\times10^{-13}$ & $1.13\times10^{-5}$ & $4.69\times10^{-12}$ & $5.47\times10^{-5}$ & 1.2 & 0.38 \\
\hline
3D & WENO & $5.38\times10^{-5}$ & $1.20\times10^{-2}$ & $9.21\times10^{-2}$ & $7.65\times10^{-1}$ & 1728.87 & 642.54 \\
   & ImplicitHJ & $1.15\times10^{-4}$ & $2.05\times10^{-4}$ & $3.20\times10^{-4}$ & $5.15\times10^{-4}$ & 177594.47 & 53.06 \\
   & Ours & $6.36\times10^{-11}$ & $2.41\times10^{-4}$ & $1.56\times10^{-9}$ & $4.60\times10^{-4}$ & 0.32 & 0.11 \\
\hline
10D & WENO & N/A & N/A & N/A & N/A & N/A & N/A \\
    & ImplicitHJ & $1.23\times10^{-3}$ & $2.45\times10^{-3}$ & $3.56\times10^{-3}$ & $6.32\times10^{-3}$ & 195088.43 & 53.60 \\
    & Ours & $2.82\times10^{-8}$ & $1.13\times10^{-3}$ & $4.69\times10^{-7}$ & $8.41\times10^{-3}$ & 1.20 & 0.38 \\
\hline
50D & WENO & N/A & N/A & N/A & N/A & N/A & N/A \\
    & ImplicitHJ & $5.12\times10^{-1}$ & $1.02\times10^{0}$ & $1.08\times10^{0}$ & $1.05\times10^{0}$ & 242050.79 & 56.70 \\
    & Ours & $1.72\times10^{-5}$ & $2.90\times10^{-2}$ & $6.25\times10^{-3}$ & $1.00\times10^{0}$ & 18.25 & 3.86 \\
\hline
100D & WENO & N/A & N/A & N/A & N/A & N/A & N/A \\
     & ImplicitHJ & $3.32\times10^{-1}$ & $6.65\times10^{-1}$ & $1.13\times10^{0}$ & $1.10\times10^{0}$ & 307335.52 & 62.61 \\
     & Ours & $7.54\times10^{-4}$ & $3.43\times10^{-2}$ & $2.15\times10^{-1}$ & $1.00\times10^{0}$ & 25.97 & 3.86 \\
\hline
\end{tabular}%
}
\end{table}

The results of our methodology with multiple initialization, compared with WENO and ImplicitHJ, are summarized in Table \ref{tab:comparison_steady_kink}. In the experiments, the number of initialization points was set to $100 \times \text{dim}$. 
From the results, we observe that the multiple-initialization approach effectively captures multiple fixed points and accurately approximates solutions exhibiting kinks. Moreover, it demonstrates better scalability compared to WENO, while being more accurate and stable than ImplicitHJ. Although the use of multiple initialization increases computational cost, our method remains faster than both baseline methods.

\subsubsection{Unsteady Kink}
We now consider a problem exhibiting an unsteady kink. Let the Hamiltonian and initial function be
 $H\left(p\right) = \frac{1}{2}\left\Vert p\right\Vert_2^2$ and initial function $g\left(\bx\right) = \sum_{i=1}^d g_0(x_i)$, where $g_0(x)=\begin{cases}
     x,&\text{ if } x<0\\
      0,&\text{ if } x\geq 0. 
      \end{cases}. $
The exact solution is given by
\[
u(x,t) = \sum_{x_i < t/2} (x_i - t/2),
\]
which exhibits a kink moving in time at \(x = t/2\). In this problem, characteristics overlap in the region \(0 < x < t\), leading to non-uniqueness of fixed points. We investigate whether the multiple initialization scheme can effectively capture such moving kinks.

The results of our multiple initialization scheme, compared to WENO and ImplicitHJ, are summarized in Table \ref{tab:comparison_unsteady_kink}. Similar to the steady kink case, the multiple-initialization fixed point approach effectively captures solutions exhibiting unsteady kinks. It demonstrates superior scalability and is more accurate and stable than both WENO and ImplicitHJ. Although the use of multiple initialization increases computational cost, our method remains faster than both baseline approaches.

\begin{table}[ht]
\centering
\caption{Comparison of numerical and gradient errors, computation time, and memory usage for different dimensions using WENO, ImplicitHJ, and Ours methods on the non-steady kink example.}
\label{tab:comparison_unsteady_kink}
\renewcommand{\arraystretch}{1.15}
\setlength{\tabcolsep}{3.5pt}
\resizebox{\textwidth}{!}{%
\begin{tabular}{c|c|cccccc}
\hline
\textbf{Dim} & \textbf{Method} & $\mathbf{L^2}$ \textbf{Error} & $\mathbf{L^\infty}$ \textbf{Error} & $\mathbf{L^2}$ \textbf{Grad Error} & $\mathbf{L^\infty}$ \textbf{Grad Error} & \textbf{Time (s)} & \textbf{Mem (MB)} \\
\hline
1D & WENO & $1.46\times10^{-7}$ & $2.02\times10^{-3}$ & $2.19\times10^{-3}$ & $2.40\times10^{-1}$ & 0.38 & 0.05 \\
   & ImplicitHJ & $5.16\times10^{-4}$ & $1.33\times10^{-1}$ & $5.57\times10^{-2}$ & $9.57\times10^{-1}$ & 158959.15 & 0.27 \\
   & Ours & $1.51\times10^{-16}$ & $1.39\times10^{-7}$ & $7.81\times10^{-13}$ & $2.10\times10^{-4}$ & 0.003 & 0.002 \\
\hline
2D & WENO & $6.84\times10^{-5}$ & $1.68\times10^{-2}$ & $2.25\times10^{-3}$ & $4.59\times10^{-1}$ & 9.072 & 4.43 \\
   & ImplicitHJ & $1.01\times10^{-4}$ & $1.04\times10^{-1}$ & $1.62\times10^{-2}$ & $9.13\times10^{-1}$ & 169483.23 & 0.34 \\
   & Ours & $3.17\times10^{-15}$ & $6.37\times10^{-6}$ & $7.81\times10^{-10}$ & $7.31\times10^{-3}$ & 0.21 & 0.08 \\
\hline
3D & WENO & $4.16\times10^{-5}$ & $1.76\times10^{-1}$ & $2.03\times10^{0}$ & $5.66\times10^{0}$ & 1734.16 & 642.54 \\
   & ImplicitHJ & $2.70\times10^{-3}$ & $4.08\times10^{-1}$ & $6.89\times10^{-2}$ & $1.08\times10^{0}$ & 185356.57 & 0.42 \\
   & Ours & $6.36\times10^{-15}$ & $2.41\times10^{-5}$ & $1.56\times10^{-2}$ & $8.27\times10^{-1}$ & 0.32 & 0.12 \\
\hline
10D & WENO & N/A & N/A & N/A & N/A & N/A & N/A \\
    & ImplicitHJ & $1.01\times10^{-4}$ & $3.02\times10^{-2}$ & $3.14\times10^{-2}$ & $6.32\times10^{-1}$ & 192581.22 & 0.95 \\
    & Ours & $2.82\times10^{-9}$ & $1.13\times10^{-3}$ & $4.69\times10^{-2}$ & $1.00\times10^{0}$ & 1.20 & 0.38 \\
\hline
50D & WENO & N/A & N/A & N/A & N/A & N/A & N/A \\
    & ImplicitHJ & $2.01\times10^{-4}$ & $5.24\times10^{-2}$ & $1.08\times10^{-1}$ & $5.48\times10^{-1}$ & 256352.28 & 4.01 \\
    & Ours & $1.72\times10^{-6}$ & $3.90\times10^{-3}$ & $6.25\times10^{-2}$ & $1.00\times10^{0}$ & 18.25 & 3.86 \\
\hline
100D & WENO & N/A & N/A & N/A & N/A & N/A & N/A \\
     & ImplicitHJ & $3.32\times10^{-1}$ & $1.57\times10^{0}$ & $1.13\times10^{0}$ & $1.57\times10^{0}$ & 321914.77 & 7.82 \\
     & Ours & $1.54\times10^{-4}$ & $3.43\times10^{-2}$ & $2.15\times10^{-1}$ & $1.00\times10^{0}$ & 25.97 & 3.86 \\
\hline
\end{tabular}%
}
\end{table}

\section{Conclusion}

We have presented a novel fixed-point iteration framework for solving high-dimensional Hamilton--Jacobi equations via the Hopf--Lax formula. By reformulating the variational problem as a fixed-point equation, the proposed method enables direct, mesh-free evaluation of viscosity solutions and associated optimal controls without requiring spatial discretization, characteristic tracing, or repeated differentiation.
Our theoretical analysis establishes convergence and error bounds under standard convexity and Lipschitz continuity assumptions, ensuring both accuracy and stability. Although the convexity requirement slightly limits the class of problems to which the method can be directly applied, numerical experiments in dimensions up to 100 demonstrate that the approach achieves high precision, low computational cost, and near dimension-independent scalability, outperforming classical grid-based solvers and recent deep learning methods in both efficiency and reliability. Overall, the fixed-point framework complements the limitations of existing classical numerical schemes and modern deep learning approaches, offering a scalable, provably accurate, and computationally efficient methodology for high-dimensional Hamilton--Jacobi problems, with potential applications in optimal control, robotics, finance, and other domains involving high-dimensional PDEs.

Several directions remain for future research. Extending the method to nonconvex or non-Lipschitz Hamiltonians would broaden its applicability to a wider class of problems. Incorporating a diffusion term to extend the method to second-order Hamilton--Jacobi equations could enable applications in stochastic control, mean-field games, and related areas. Furthermore, incorporating adaptive step-size control, momentum, anchoring, Anderson acceleration, or quasi-Newton updates may further enhance convergence speed and robustness, particularly in very high-dimensional settings. Finally, combining the fixed-point approach with implicit neural network architectures \cite{bai2019deep,fung2022jfb} could open new possibilities for solving inverse problems and other complex high-dimensional tasks, thereby further extending the practical impact of the method.

\bibliographystyle{siamplain}
\bibliography{mybib}

\end{document}